\documentclass[12pt,a4paper,reqno,twoside]{amsart}

\usepackage[english]{babel}
\usepackage{stmaryrd}
\usepackage{dsfont}
\usepackage[symbol*,ragged]{footmisc}
\usepackage[colorlinks,linkcolor=red,anchorcolor=blue,citecolor=blue,urlcolor=blue]{hyperref}
\usepackage{color,xcolor}

\usepackage{geometry}
\usepackage{amssymb}
\usepackage{amsmath}
\usepackage{mathrsfs}
\usepackage{amsfonts}
\usepackage{epsfig}

\usepackage{amsthm}
\usepackage{amsxtra}
\usepackage{bbding}
\usepackage{epsfig}
\usepackage{graphicx}
\usepackage{latexsym}
\usepackage{mathbbol}
\usepackage{bbold}

\usepackage{pifont}
\usepackage{wasysym}
\usepackage{skull}
\usepackage{float}

\DeclareSymbolFontAlphabet{\mathbb}{AMSb}
\DeclareSymbolFontAlphabet{\mathbbol}{bbold}

\usepackage{amscd}
\usepackage[all]{xy}
\allowdisplaybreaks[4]
\usepackage{setspace}

\theoremstyle{plain}
\newtheorem{theorem}{\normalfont\scshape Theorem}[section]
\newtheorem{proposition}{\normalfont\scshape Proposition}[section]
\newtheorem{lemma}[proposition]{\normalfont\scshape Lemma}
\newtheorem{corollary}[theorem]{\normalfont\scshape Corollary}
\newtheorem*{corollary*}{\normalfont\scshape Corollary}

\theoremstyle{remark}
\newtheorem*{remark*}{\normalfont\scshape Remark}
\newtheorem*{notation}{\normalfont\scshape Notation}

\numberwithin{equation}{section}
\addtocounter{footnote}{1}

\renewcommand{\footnoterule}{
  \kern -3pt
  \hrule width 2.5in height 0.4pt
  \kern 3pt
}

\makeatletter
\@ifundefined{MakeUppercase}{}{}
\makeatother

\begin{document}
	
\title[ Cubic Waring--Goldbach problem with Piatetski--Shapiro primes ]
	  { Cubic Waring--Goldbach problem with Piatetski--Shapiro primes }

\author[Linji Long, Jinjiang Li, Min Zhang, Yankun Sui]
       {Linji Long \quad \& \quad Jinjiang Li \quad \& \quad Min Zhang \quad \& \quad  Yankun Sui}

\address{Department of Mathematics, China University of Mining and Technology,
         Beijing 100083, People's Republic of China}

\email{linji.long.math@gmail.com}

\address{(Corresponding author) Department of Mathematics, China University of Mining and Technology,
         Beijing 100083, People's Republic of China}

\email{jinjiang.li.math@gmail.com}

\address{School of Applied Science, Beijing Information Science and Technology University,
		 Beijing 100192, People's Republic of China  }

\email{min.zhang.math@gmail.com}

\address{College of Information Engineering, Nanjing Xiaozhuang University, Nanjing 211171,  Jiangsu
          People's Republic of China}

\email{yankun.sui.math@gmail.com}

\date{}

\footnotetext[1]{Jinjiang Li is the corresponding author. 

\quad\,\,
{\textbf{MR(2020) Subject Classification}}: 11P05, 11P32, 11P55, 11L07, 11L20

}

\begin{abstract}
In this paper, it is proved that, for $\gamma\in(\frac{317}{320},1)$, every sufficiently large odd integer can be written as the sum of nine cubes of primes, each of which is of the form $[n^{1/\gamma}]$. This result constitutes an improvement upon the previous result of Akbal and G\"{u}lo\u{g}lu \cite{Akbal-Guloglu-2018}.
\end{abstract}

\maketitle

\section{Introduction and main result}
The famous Waring--Goldbach problem in additive number theory states that every large integers $N$ satisfying appropriate congruent conditions should be represented as the sum of $s$ $k$--th powers of prime numbers, i.e.,
\begin{equation}\label{Waring-Goldbach-General}
N=p_1^k+p_2^k+\dots+p_s^k.
\end{equation}
A formal application of the Hardy--Littlewood method suggests that whenever $s$ and $k$ are natural numbers with
$s\geqslant k+1$, then (\ref{Waring-Goldbach-General}) holds with the expected asymptotic formula (weighted form)
\begin{equation}\label{Waring-Goldbach-asymp}
\sum_{N=p_1^k+\dots+p_s^k}\prod_{j=1}^s\log p_j=\mathfrak{S}_{k,s}(N)\frac{\Gamma^s(1+1/k)}{\Gamma(s/k)}
N^{s/k-1}+O\bigg(\frac{N^{s/k-1}}{\log^{A}N}\bigg).
\end{equation}
With this expectation in mind, consider a natural number $k$ and prime number $p$, and define $\theta=\theta(k,p)$ to be the integer with $p^\theta|k$ but
$p^{\theta+1}\nmid k$, and $\mathfrak{s}=\mathfrak{s}(k,p)$ by
\begin{equation*}
\mathfrak{s}(k,p)=
\begin{cases}
\theta+2, & \textrm{if $p=2$ and $\theta>0$,}  \\
\theta+1, & \textrm{otherwise.}
\end{cases}
\end{equation*}
We then put
\begin{equation*}
 K(k)=\prod_{(p-1)|k}p^{\mathfrak{s}},
\end{equation*}
and denote by $H(k)$ the least integer $s$ such that every sufficiently large positive integer congruent to $s$ modulo $K(k)$ may be written as in the shape of (\ref{Waring-Goldbach-General}), with $p_1,\dots,p_s$ prime numbers. We note that the local conditions introduced in the definition of $H(k)$ are designed to exclude the degenerate situations in which one or more variables might, otherwise, be forced to be prime divisors of $K(k)$.
In such circumstances, the representation problem at hand reduces to a similar problem of Waring--Goldbach type in fewer variables.

The first general bound for $H(k)$ was obtained by Hua \cite{Hua-1938}, who showed that
\begin{equation}\label{Hua-upper-bound}
H(k)\leqslant2^k+1, \qquad (k\geqslant1).
\end{equation}
This result, which generalizes Vinogradov's celebrated three primes theorem \cite{Vinogradov-1937}, remains
the best known bound on $H(k)$ for $k=1,2,3$. When $k\geqslant4$, on the other hand, the
bound (\ref{Hua-upper-bound}) has been sharpened considerably.

Let $\gamma\in(\frac{1}{2},1)$ be a fixed real number. The Piatetski--Shapiro sequences are sequences of the form
\begin{equation*}
 \mathscr{N}_{\gamma}:=\big\{[n^{1/\gamma}]:\,n\in \mathbb{N}^+\big\}.
\end{equation*}
Such sequences have been named in honor of Piatetski--Shapiro, who \cite{Piatetski-Shapiro-1953}, in
1953, proved that $\mathscr{N}_{\gamma}$ contains infinitely many primes provided that $\gamma\in(\frac{11}{12},1)$. The prime numbers of the form $p=[n^{1/\gamma}]$ are called \textit{Piatetski--Shapiro primes of type $\gamma$}. More precisely, for such $\gamma$ Piatetski--Shapiro \cite{Piatetski-Shapiro-1953} showed that the counting function
\begin{equation*}
 \pi_\gamma(x):=\#\big\{\textrm{prime}\,\, p\leqslant x:\,p=[n^{1/\gamma}]\,\,\textrm{for some}\,\,
 n\in\mathbb{N}^+ \big\}
\end{equation*}
satisfies the asymptotic property
\begin{equation*}
\pi_{\gamma}(x)=\frac{x^{\gamma}}{\log x}(1+o(1))
\end{equation*}
as $x\to\infty$. Since then, the range for $\gamma$ of the above asymptotic formula in which it is known that $\mathscr{N}_{\gamma}$ contains infinitely many primes has been enlarged many times (e.g., see the literatures \cite{Kolesnik-1967,Leitmann-1975,Leitmann-1980,Heath-Brown-1983,Kolesnik-1985,Liu-Rivat-1992,
Rivat-1992,Rivat-Sargos-2001}) over the years and is currently known to hold for all $\gamma\in(\frac{2426}{2817},1)$ thanks to Rivat and Sargos \cite{Rivat-Sargos-2001}. Rivat and Wu \cite{Rivat-Wu-2001} also showed that there exist infinitely many Piatetski--Shapiro primes for $\gamma\in(\frac{205}{243},1)$ by showing a lower bound of $\pi_\gamma(x)$ with the expected order of magnitude. We remark that if $\gamma>1$ then $\mathscr{N}_\gamma$ contains all natural numbers, and hence all primes, particularly.

Based on the previous result, it is natural to investigate the Waring--Goldbach problem with prime variables restricted to Piatetski--Shapiro set. For the linear case, in 1992, Balog and Friedlander \cite{Balog-Friedlander-1992} firstly found an asymptotic formula for the number of solutions of the equation (\ref{Waring-Goldbach-General}) with three variables restricted to the Piatetski--Shapiro primes. An interesting corollary of their theorem is that every sufficiently large odd integer can be written as the sum of three Piatetski--Shapiro primes of type $\gamma$, provided that $\gamma\in(\frac{20}{21},1)$. Afterwards, their studies in this direction were subsequently continued by Jia \cite{Jia-1995} and by Kumchev \cite{Kumchev-1997}, and generalized by Cui \cite{Cui-2004} and Li and Zhang \cite{Li-Zhang-2018}, consecutively and respectively.

In 1998, Zhai \cite{Zhai-1998} considered the hybrid problem of quadratic Waring--Goldbach problem with each prime variable restricted to Piatetski--Shapiro sets. To be specific, he proved that, for $\gamma\in(\frac{43}{44},1)$ fixed, every sufficiently large integer $N$ satisfying $N\equiv5\pmod{24}$ can be written as five squares of primes with each prime of the form $[n^{1/\gamma}]$. Later, in 2005, Zhang and Zhai \cite{Zhang-Zhai-2005} improved the result of Zhai \cite{Zhai-1998} and enlarge the range to $\frac{249}{256} <\gamma<1$.

Afterwards, in 2018, Akbal and G\"{u}lo\u{g}lu \cite{Akbal-Guloglu-2018} considered the case $k\geqslant3$ and the solvability of (\ref{Waring-Goldbach-General}) with each prime variable restricted to Piatetski--Shapiro sets. They show that the number of representations $\mathcal{R}_{k,s}^{(\gamma)}(N)$ of a positive integer $N$ as in (\ref{Waring-Goldbach-General}) satisfies
\begin{equation*}
\mathcal{R}_{k,s}^{(\gamma)}(N)=\mathfrak{S}_{k,s}(N)\frac{\Gamma^s(1+\gamma/k)}{\Gamma(\gamma s/k)}
\frac{N^{\gamma s/k-1}}{\log^s N}+o\bigg(\frac{N^{\gamma s/k-1}}{\log^s N}\bigg),
\end{equation*}	
where $\mathfrak{S}_{k,s}(N)$ defined in (\ref{Waring-Goldbach-asymp}) is the singular series in the classical Waring--Goldbach problem, provided that $\gamma$ is a fixed number satisfying $\gamma_k<\gamma<1$ with explicit value for each $k\geqslant3$. When $k=3$ and $s=9$, Akbal and G\"{u}lo\u{g}lu \cite{Akbal-Guloglu-2018} showed that
$\gamma_3=\frac{1331}{1334}$, which means every sufficiently large odd integer $N$ can be written as nine cubes of primes with each prime of the form $[n^{1/\gamma}]$, provided that $\frac{1331}{1334}<\gamma<1$.

In this paper, we shall improve the result of Akbal and G\"{u}lo\u{g}lu \cite{Akbal-Guloglu-2018} for $k=3$ and $s=9$, and establish the following theorem.

\begin{theorem}\label{Theorem-1}
Let $\gamma_1,\gamma_2,\dots,\gamma_9$ be fixed subject to $1/2<\gamma_i<1$. Suppose that for $i=1,2,\dots,9$, there holds
\begin{equation*}
\frac{80}{3}(1-\gamma_i)+\frac{80}{3}\Delta_i<1,
\end{equation*}
where
\begin{equation*}
\Delta_i =
\begin{cases}
\displaystyle\frac{3}{8}\sum\limits_{j=i+1}^9(1-\gamma_j), & \textrm{if $1\leqslant i\leqslant 8$,}  \\
\hspace{3em} 0, & \textrm{if \hspace{0.7em} $i=9$.}
\end{cases}
\end{equation*}
Then, for any $A>0$ and with $\mathfrak{S}_{3,9}(N)$ defined by (\ref{Waring-Goldbach-asymp}), one has
\begin{align}\label{eq1_01}
          \mathscr{T}(N)
:= & \,\, \frac{1}{\gamma_1\gamma_2\cdots\gamma_9}
          \sum_{\substack{N=p_1^3+p_2^3+\dots+p_9^3 \\ p_i=[n^{1/\gamma_i}] \\ i=1,2,\dots,9 }}
          p_1^{1-\gamma_1}(\log p_1)p_2^{1-\gamma_2}(\log p_2)\cdots p_9^{1-\gamma_9}(\log p_9)
                 \nonumber \\
  = & \,\, \frac{\Gamma^9(4/3)}{\Gamma(3)} \mathfrak{S}_{3,9}(N){N^{2}}+O\bigg(\frac{N^{2}}{\log^A N}\bigg).
\end{align}
\end{theorem}
This means in particular that we may require the nine summands to be Piatetski--Shapiro primes of distinct type. However, by choosing $\gamma_1=\gamma_2=\dots=\gamma_9=\gamma$ we obtain the following conclusion.
\begin{corollary}\label{Corollary1_2}
For any fixed $317/320 <\gamma < 1$, the primes $p$ of the form $[n^{1/\gamma}]$ have the property that every sufficiently large odd integer can be written as the sum of nine cubes of them.
\end{corollary}
Also, by choosing $\gamma_1=\gamma_2=\cdots=\gamma_8=1$ in Theorem \ref{Theorem-1}, we derive another corollary.
\begin{corollary}\label{Corollary1_3}
For any fixed $77/80<\gamma<1$, every sufficiently large odd integer can be written as the sum of nine cubes of primes, one of which is of the form $[n^{1/\gamma}]$.
\end{corollary}

\begin{remark*}
In order to compare our result with the previous result of Akbal and G\"{u}lo\u{g}lu \cite{Akbal-Guloglu-2018}, we list the numerical result as follows
\begin{equation*}
\frac{317}{320}=0.990625,\qquad \qquad \frac{1331}{1334}=0.997751\dots.
\end{equation*}
\end{remark*}

\begin{notation}
Throughout this paper, $N$ is a sufficiently large integer, $\varepsilon$ is a sufficiently small positive number. Let $p$, with or without subscripts, always denote a prime number. We use $[t],\{t\}$ and $\|t\|$ to denote the integral part of $t$, the fractional part of $t$ and the distance from $t$ to the nearest integer, respectively. As usual, $\Lambda(n),\mu(n)$ and $\tau_j(n)$ denote von Mangoldt's function, Mobius' function, and the $j$--dimensional divisor function, respectively. Especially, we denote by $\tau(n)=\tau_2(n)$ the classical Dirichlet's divisor function.
We write $\psi(t)=t-[t]-1 / 2, e(t)=\exp(2\pi it)$. The notation $n \sim Y$ means that $n$ runs through a subinterval of $(Y, 2 Y]$, whose endpoints are not necessarily the same in the different occurrences and may depend on the outer summation variables. $f(x) \ll g(x)$ means that $f(x)=O(g(x))$; $f(x) \asymp g(x)$ means that $f(x)\ll g(x)\ll f(x)$.
\end{notation}


\section{Preliminary Lemmas}

In this section, we shall demonstrate some lemmas which are necessary for establishing Theorem \ref{Theorem-1}.

\begin{lemma}\label{Intersection-of-PS-Primes}
Suppose that $11/12<\gamma<1$. Then one has
\begin{equation*}
\pi_\gamma(x):=\sum_{\substack{p\leqslant x\\ p=[n^{1/\gamma}]}}1=\frac{x^{\gamma}}{\log x}(1+o(1)).
\end{equation*}
\end{lemma}
\begin{proof}
See Theorem 2 on pp. 565--566 of Piatetski--Shapiro \cite{Piatetski-Shapiro-1953}.
\end{proof}

\begin{lemma}\label{Lemma2_2}
Assume that $11/12<\gamma<1$ and $x$ is a sufficiently large positive number. Define
\begin{align*}
\mathscr{P}=\big\{p:p\leqslant x,\,p=[n^{1/\gamma}]\big\}, \qquad
f(\alpha)=\sum_{p\in\mathscr{P}}e(p^3\alpha).
\end{align*}
Then, we have
\begin{equation*}
\int_0^1|f(\alpha)|^8\mathrm{d}\alpha\ll|\mathscr{P}|^{5+\varepsilon},
\end{equation*}
where $|\mathscr{P}|$ denotes the cardinality of $\mathscr{P}$.
\end{lemma}
\begin{proof}
First, it is easy to see that
\begin{align*}
          |f(\alpha)|^2
 = & \,\, \sum_{p_1\in\mathscr{P}}\sum_{p_2\in\mathscr{P}}e((p_1^3-p_2^3)\alpha)
                   \nonumber \\
 = & \,\, \sum_k\sum_{\substack{p\in\mathscr{P} \\ p+k \in \mathscr{P}}}e(((p+k)^3-p^3)\alpha)
                   \nonumber \\
 = & \,\,  \sum_j e(j\alpha)\Bigg(\mathop{\sum_k \sum_{\substack{p\in\mathscr{P} \\ p+k \in \mathscr{P} \\
           k(3p^2+3pk+k^2)=j}}}1 \Bigg) \\
 =: & \,\, \sum_j c_j e(j\alpha),
\end{align*}
where $c_j$ is the number of solutions of the equation
\begin{equation*}
k(3p^2+3pk+k^2)=j
\end{equation*}
with $p,p+k\in \mathscr{P}$ and $\sum_k 1\ll |\mathscr{P}|$. Clearly, for $j=0$, there must hold $k=0$. Otherwise,
if $k\not=0$, by noting the fact that the discriminant of the quadratic polynomial in $p$ is $-3k^2<0$, one has
\begin{equation*}
3p^2+3pk+k^2>0,
\end{equation*}
which contradicts to $j=0$. Hence, we get $c_0\ll|\mathscr{P}|$. For $j \neq 0$, we derive that
\begin{equation*}
c_j\ll\tau(j)\ll\tau\big(\big|k(3p^2+3pk+k^2)\big|\big)\ll|\mathscr{P}|^\varepsilon.
\end{equation*}
On the other hand, one obtains
\begin{align*}
          |f(\alpha)|^2
 =  & \,\, \sum_{p_1\in\mathscr{P}}\sum_{p_2\in\mathscr{P}}e\big((p_1^3-p_2^3)\alpha\big) \\
 =  & \,\, \sum_je(j\alpha)\mathop{\sum_{p_1\in\mathscr{P}}\sum_{p_2\in\mathscr{P}}}_{p_1^3-p_2^3=j} 1 \\
 =: & \sum_j b_j e(j\alpha),
\end{align*}
where $b_j$ is the number of solutions of the equation
\begin{equation*}
p_1^3-p_2^3=j
\end{equation*}
with $p_1,p_2\in \mathscr{P}$. Trivially, one has
\begin{equation*}
\sum_jb_j=|f(0)|^2=|\mathscr{P}|^2,
\end{equation*}
and
\begin{equation*}
b_0=\int_0^1|f(\alpha)|^2\mathrm{d}\alpha=|\mathscr{P}|.
\end{equation*}
Hence, according to Parseval's identity, we derive that
\begin{align}
          \int_0^1|f(\alpha)|^4\mathrm{d}\alpha
 = & \,\, \sum_jb_jc_j=b_0c_0+\sum_{j\neq0}b_jc_j
              \nonumber \\
\ll & \,\, |\mathscr{P}|^2+|\mathscr{P}|^\varepsilon\sum_{j\neq 0}b_j\ll |\mathscr{P}|^{2+\varepsilon}. \label{eq2_1}
\end{align}
Similarly, by Cauchy's inequality, the fourth power of $|f|$ is
\begin{align*}
           |f(\alpha)|^4
\ll & \,\, \bigg|\sum_{k_1}\sum_{\substack{p\in\mathscr{P}\\ p+k_1 \in \mathscr{P}}}
           e\big( k_1 (3p^2+3pk_1+k_1^2)\alpha\big)\bigg|^2
                 \nonumber \\
\ll & \,\, |\mathscr{P}|\cdot\sum_{k_1}\bigg|\sum_{\substack{p\in\mathscr{P}\\ p+k_1\in\mathscr{P}}}
           e\big(k_1 (3p^2+3pk_1+k_1^2)\alpha\big) \bigg|^2
                   \nonumber \\
= & \,\,  |\mathscr{P}| \cdot \sum_{k_1}\sum_{\substack{p_1\in\mathscr{P}\\p_1+k_1 \in \mathscr{P}}}
          \sum_{\substack{p_2\in\mathscr{P} \\p_2+k_1 \in \mathscr{P}}}
          e\big( k_1 (3p_1^2+3p_1k_1-3p_2^2-3p_2k_1)\alpha\big)
                   \nonumber \\
= & \,\, |\mathscr{P}|\cdot\sum_{k_1}\sum_{k_2}
         \sum_{\substack{p,\, p+k_1\in\mathscr{P} \\p+k_2,\, p+k_1+k_2 \in \mathscr{P}}}
          e\big(3k_1k_2 (2p+k_1+k_2)\alpha \big) \\
= & \,\, |\mathscr{P}| \cdot \sum_j e(j\alpha) \Bigg(\sum_{k_1} \sum_{k_2} \sum_{\substack{p,\, p+k_1\in\mathscr{P} \\p+k_2,\, p+k_1+k_2 \in \mathscr{P} \\ 3k_1k_2(2p+k_1+k_2)=j}} 1 \Bigg)\\
=: & \,\, |\mathscr{P}| \cdot\sum_j c^*_j e(j\alpha),
\end{align*}
where $c^*_j$ is the number of solutions of the equation
\begin{equation*}
3k_1k_2(2p+k_1+k_2)=j
\end{equation*}
with $p,p+k_1,p+k_2,p+k_1+k_2\in \mathscr{P}$, $\sum_{k_1} 1\ll |\mathscr{P}|$ and $\sum_{k_2} 1\ll |\mathscr{P}|$. Trivially, for $j=0$, by noting that $p+k_1,p+k_2\in \mathscr{P}$, we get $k_1=0$ or $k_2=0$. Thus,
\begin{equation*}
c_0^*\ll |\mathscr{P}|^2.
\end{equation*}
For $j \neq 0$, we have
\begin{equation*}
c_j^*\ll \tau_4(j)\ll\tau_4\big(|3k_1k_2(2p+k_1+k_2)|\big)\ll|\mathscr{P}|^\varepsilon.
\end{equation*}
Similarly, one also has
\begin{align*}
          |f(\alpha)|^4
 = & \,\, \sum_{p_1\in\mathscr{P}}\sum_{p_2\in\mathscr{P}}\sum_{p_3\in\mathscr{P}}
          \sum_{p_4\in\mathscr{P}} e\big((p_3^3+p_4^3-p_1^3-p_2^3)\alpha\big)
                   \nonumber \\
 = & \,\, \sum_j e(-j\alpha)\mathop{\sum_{p_1\in\mathscr{P}}\sum_{p_2\in\mathscr{P}}\sum_{p_3\in\mathscr{P}}
          \sum_{p_4\in\mathscr{P}}}_{p_1^3+p_2^3-p_3^3-p_4^3=j} 1
                   \nonumber \\
= :& \,\, \sum_j b_j^* e(-j\alpha),
\end{align*}
where $b_j^*$ is the number of solutions of the equation
\begin{equation*}
p_1^3+p_2^3-p_3^3-p_4^3=j
\end{equation*}
with $p_1,p_2,p_3,p_4\in \mathscr{P}$. Clearly,
\begin{equation*}
\sum_j b_j^*=|f(0)|^4=|\mathscr{P}|^4.
\end{equation*}
According to \eqref{eq2_1}, we derive that
\begin{equation*}
b_0^* = \int_0^1 |f(\alpha)|^4 \mathrm{d} \alpha \ll |\mathscr{P}|^{2+\varepsilon}.
\end{equation*}
Thus, according to Parseval's identity, we deduce that
\begin{align*}
\int_0^1 |f(\alpha)|^8 \mathrm{d} \alpha & \ll |\mathscr{P}| \sum_j b_j^* c_j^* = |\mathscr{P}| b_0^* c_0^* +|\mathscr{P}|\sum_{j\neq 0} b_j^* c_j^*\\
& \ll |\mathscr{P}|^{5+\varepsilon} + |\mathscr{P}|^{1+\varepsilon} \sum_{j\neq 0} b_j^* \ll |\mathscr{P}|^{5+\varepsilon}.
\end{align*}
This completes the proof of Lemma \ref{Lemma2_2}.
\end{proof}

\begin{lemma}\label{Hua's-inequality}
Suppose that $1\leqslant j\leqslant k$, and let
\begin{equation*}
g(\alpha) = \sum_{m=1}^Y e(m^k\alpha).
\end{equation*}
Then we have
\begin{equation*}
\int_0^1 |g(\alpha)|^{2^j} \mathrm{d} \alpha \ll Y^{2^j-j+\varepsilon}.
\end{equation*}
\end{lemma}
\begin{proof}
See Lemma 2.5 of Vaughan \cite{Vaughan-book}.
\end{proof}

\begin{lemma}\label{Finite-Fourier-expansion}
For any $H>1$, one has
\begin{equation*}
  \psi(\theta)=-\sum_{0<|h|\leqslant H}\frac{e(\theta h)}{2\pi ih}+O(g(\theta,H)),
\end{equation*}
where
\begin{equation*}
g(\theta,H)=\min\bigg(1,\frac{1}{H\|\theta\|}\bigg)
=\sum_{h=-\infty}^{\infty}a(h)e(\theta h),
\end{equation*}
\begin{equation*}
a(0)\ll \frac{\log 2H}{H},\ \
a(h)\ll \min\bigg(\frac{1}{|h|},\frac{H}{h^2}\bigg)\quad (h\neq 0).
\end{equation*}
\end{lemma}
\begin{proof}
See the arguments on page 245 of Heath--Brown \cite{Heath-Brown-1983}.
\end{proof}

\begin{lemma}\label{vander-Corput-method}
Suppose that $5<A<B\leqslant 2A$, $f''(x)$ is continuous on $[A, B]$. If there exist two positive constants $c_1$ and $c_2$ such that $0<c_1 \lambda_2 \leqslant|f^{\prime \prime}(x)| \leqslant c_2 \lambda_2$, then
\begin{equation*}
\sum_{A<n \leqslant B}e(f(n))\ll A\lambda_2^{1/2}+\lambda_2^{-1/2}.
\end{equation*}
\end{lemma}
\begin{proof}
See Theorem 2.2 of Graham and Kolesnik \cite{Graham-Kolesnik-book}.
\end{proof}

\begin{lemma}\label{Van-der-Corput-k-derivatives}
Let $q$ be a positive integer. Suppose that $f$ is a real valued function with $q+2$ continuous derivatives on some interval $I$. Suppose also that for some $\lambda>0$ and for some $\alpha>1$, there holds
\begin{equation*}
\lambda\leqslant \big|f^{(q+2)}(x)\big|\leqslant \alpha \lambda
\end{equation*}
on I. Let $Q=2^q$. Then one has
\begin{equation*}
\sum_{n \in I} e(f(n)) \ll |I|(\alpha^2\lambda)^{1/(4Q-2)}+|I|^{1-1/(2Q)}\alpha^{1/(2Q)}+|I|^{1-2/Q+1/Q^2}\lambda^{-1/(2Q)}.
\end{equation*}
where the implied constant is absolute.
\end{lemma}
\begin{proof}
See Theorem 2.8 of Graham and Kolesnik \cite{Graham-Kolesnik-book}.
\end{proof}

\begin{lemma}\label{Heath-Brown-1983}
For $1/2<\alpha<1$, $H\geqslant 1$, $K\geqslant 1$, $\Delta>0$, let $\mathscr{N}(\Delta)$ denote the number of solutions of the following inequality
\begin{equation*}
|h_1k_1^\alpha - h_2k_2^\alpha|\leqslant \Delta, \quad h_1,h_2\sim H, \quad  k_1,k_2\sim K.
\end{equation*}
Then we have
\begin{equation*}
\mathscr{N}(\Delta)\ll\Delta HK^{2-\alpha}+HK\log(HK).
\end{equation*}
\end{lemma}
\begin{proof}
See the arguments on pp. 256--257 of Heath--Brown \cite{Heath-Brown-1983}.
\end{proof}

\begin{lemma}\label{Optimization-Principle}
Suppose that
\begin{equation*}
L(q) = \sum_{i=1}^{m} A_i q^{u_i} + \sum_{j=1}^{n} B_j q^{-v_j},
\end{equation*}
where $m,n,A_i,B_j,u_i$ and $v_j$ are positive. Assume further that $0\leqslant Q_1\leqslant Q_2$. Then there exists some $q$ subject to $Q_1\leqslant q\leqslant Q_2$ such that
\begin{equation*}
L(q)\ll\sum_{i=1}^{m}\sum_{j=1}^{n}(A_i^{v_j}B_j^{u_i})^{1/(u_i+v_j)}+\sum_{i=1}^{m}A_iQ_1^{u_i}
+\sum_{i=1}^{n} B_j Q_2^{-v_j}.
\end{equation*}
The implied constants depend only on $m$ and $n$.
\end{lemma}
\begin{proof}
See Lemma 3 of Srinivasan \cite{Srinivasan-1963}.
\end{proof}

\begin{lemma}\label{Heath-Brown-identity}
  Let $z\geqslant1$ and $k\geqslant1$. Then, for any $n\leqslant2z^k$, there holds
\begin{equation*}
\Lambda(n)=\sum_{j=1}^k(-1)^{j-1}\binom{k}{j}\mathop{\sum\cdots\sum}_{\substack{n_1n_2\cdots n_{2j}=n\\
n_{j+1},\dots,n_{2j}\leqslant z }}(\log n_1)\mu(n_{j+1})\cdots\mu(n_{2j}).
\end{equation*}
\end{lemma}
\begin{proof}
See the arguments on pp. 1366--1367 of Heath--Brown \cite{Heath-Brown-1982}.
\end{proof}

\section{ Translation of the problem}
Let $11/12<\gamma_i<1\,(i=1,2,\dots,9)$ be real numbers. Define
\begin{equation*}
G(\alpha)=\sum_{p\leqslant N^{1/3}}(\log p)e(p^3\alpha),
\end{equation*}
and
\begin{equation*}
F_i(\alpha)=\frac{1}{\gamma_i} \sum_{\substack{p\leqslant N^{1/3} \\ p=[n^{1/\gamma_i}]}}
p^{1-\gamma_i}(\log p)e(p^3\alpha), \qquad (i=1,2,\dots,9).
\end{equation*}
By the orthogonality of exponential function, we get
\begin{equation*}
\mathscr{R}(N):=\sum_{N=p_1^3+p_2^3+\dots+ p_9^3}\prod_{j=1}^9\log p_j=\int_0^1 G^9(\alpha)
e(-N\alpha)\mathrm{d}\alpha,
\end{equation*}
and the sum on the left-hand side of \eqref{eq1_01} can be written as
\begin{equation*}
\mathscr{T}(N)=\int_0^1 F_1(\alpha) F_2(\alpha) \cdots F_9(\alpha) e(-N \alpha) \mathrm{d} \alpha .
\end{equation*}
It is easy to see that
\begin{align}\label{eq3_1}
          F_1F_2\cdots F_9-G^9
 = & \,\, (F_1-G)F_2\cdots F_9+G(F_2-G)F_3\cdots F_9+G^2(F_3-G)F_4\cdots F_9
                  \nonumber \\
   & \,\, + G^3(F_4-G)F_5\cdots F_9+G^4(F_5-G)F_6\cdots F_9+G^5(F_6-G)F_7F_8F_9
                  \nonumber \\
   & \,\, + G^6(F_7-G)F_8F_9+G^7(F_8-G)F_9+G^8(F_9-G).
\end{align}
Thus, by \eqref{eq3_1}, we have $\mathscr{T}(N)=\mathscr{R}(N)+\mathscr{E}$, where
\begin{align}\label{eq3_2}
            \mathscr{E}
 \ll & \,\, \sup_{\alpha\in[0,1]}|F_1-G|\times\int_0^1|F_2\cdots F_9|\mathrm{d}\alpha
            +\sup_{\alpha\in[0,1]}|F_2-G|\times\int_0^1|GF_3\cdots F_9|\mathrm{d}\alpha
                   \nonumber \\
     & +\sup_{\alpha\in[0,1]}|F_3-G|\times\int_0^1|G^2 F_4\cdots F_9|\mathrm{d}\alpha
       +\sup_{\alpha\in(0,1)}|F_4-G|\times\int_0^1|G^3F_5\cdots F_9|\mathrm{d}\alpha
                   \nonumber \\
     & +\sup_{\alpha\in[0,1]}|F_5-G|\times\int_0^1|G^4F_6\cdots F_9|\mathrm{d}\alpha
       +\sup_{\alpha\in[0,1]}|F_6-G|\times\int_0^1|G^5F_7\cdots F_9|\mathrm{d}\alpha
                   \nonumber \\
     & +\sup_{\alpha\in[0,1]}|F_7-G|\times\int_0^1|G^6F_8F_9|\mathrm{d}\alpha
       +\sup_{\alpha\in[0,1]}|F_8-G|\times\int_0^1|G^7F_9|\mathrm{d}\alpha
                   \nonumber \\
     & +\sup_{\alpha\in[0,1]}|F_9-G|\times\int_0^1|G^8|\mathrm{d}\alpha.
\end{align}
Hence, it follows from partial summation that
\begin{equation*}
F_i(\alpha)\ll N^{(1-\gamma_i)/3}(\log N)\sum_{\substack{p\leqslant N^{1/3}\\ p=[n^{1/\gamma_i}]}}
e(p^3\alpha),
\end{equation*}
which combined with Lemma \ref{Intersection-of-PS-Primes} and Lemma \ref{Lemma2_2} yields
\begin{align}\label{eq3_3}
         \int_0^1|F_i(\alpha)|^8\mathrm{d}\alpha
= & \,\, N^{8(1-\gamma_i)/3}(\log N)^8\int_0^1\Bigg|\sum_{\substack{p\leqslant N^{1/3}\\ p=[n^{1/\gamma_i}]}}
         e(p^3\alpha)\Bigg|^8\mathrm{d}\alpha
                \nonumber \\
\ll & \,\, N^{8(1-\gamma_i)/3}(\log N)^8\bigg(\frac{N^{\gamma_i/3}}{\log N}\bigg)^{5+\varepsilon}
                \nonumber \\
\ll & \,\, N^{8/3-\gamma_i+\varepsilon}.
\end{align}
Therefore, we can give the upper bound estimate of the nine integrals above in \eqref{eq3_2} by H\"{o}lder's inequality, Lemma \ref{Hua's-inequality} and \eqref{eq3_3}, and deduce that
\begin{align*}
         \int_0^1|G|^8\mathrm{d}\alpha
= & \,\, \int_0^1\bigg|\sum_{p\leqslant N^{1/3}}(\log p)e(p^3\alpha)\bigg|^8\mathrm{d}\alpha
         \ll N^{5/3+\varepsilon},
                 \nonumber \\
            \int_0^1|G^7F_9|\mathrm{d}\alpha
 \ll & \,\, \bigg(\int_0^1|G|^8\mathrm{d}\alpha\bigg)^{7/8}
            \bigg(\int_0^1|F_9|^8\mathrm{d}\alpha\bigg)^{1/8}
            \ll N^{{5}/{3}+(1-\gamma_9)/8+\varepsilon},
                  \nonumber \\
            \int_0^1|G^6F_8F_9|\mathrm{d}\alpha
 \ll & \,\, \bigg(\int_0^1|G|^8\mathrm{d}\alpha\bigg)^{3/4}
            \bigg(\int_0^1|F_8|^8\mathrm{d}\alpha\bigg)^{1/8}
            \bigg(\int_0^1|F_9|^8\mathrm{d}\alpha\bigg)^{1/8}
                  \nonumber \\
 \ll & \,\, N^{{5}/{3}+(1-\gamma_8)/8+(1-\gamma_9)/8+\varepsilon},
                  \nonumber \\
            \int_0^1 |G^5 F_7 F_8 F_9| \mathrm{d}\alpha
 \ll & \,\, \bigg(\int_0^1|G|^8\mathrm{d}\alpha \bigg)^{5/8}
            \bigg(\int_0^1 |F_7|^8 \mathrm{d} \alpha \bigg)^{1/8}
            \bigg(\int_0^1 |F_8|^8 \mathrm{d} \alpha \bigg)^{1/8}
            \bigg(\int_0^1 |F_9|^8 \mathrm{d} \alpha \bigg)^{1/8}
                  \nonumber \\
 \ll & \,\, N^{\frac{5}{3}+\frac{1}{8}\sum_{j=7}^9(1-\gamma_j)+\varepsilon}.
\end{align*}
Similarly, we can obtain
\begin{gather*}
\int_0^1 |G^4 F_6\cdots F_9| \mathrm{d}\alpha  \ll N^{\frac{5}{3}+\frac{1}{8}\sum_{j=6}^9(1-\gamma_j)+\varepsilon}, \\
\int_0^1 |G^3 F_5\cdots F_9| \mathrm{d}\alpha  \ll N^{\frac{5}{3}+\frac{1}{8}\sum_{j=5}^9(1-\gamma_j)+\varepsilon}, \\
\int_0^1 |G^2 F_4\cdots F_9| \mathrm{d}\alpha  \ll N^{\frac{5}{3}+\frac{1}{8}\sum_{j=4}^9(1-\gamma_j)+\varepsilon}, \\
\int_0^1 |G F_3\cdots F_9| \mathrm{d}\alpha  \ll N^{\frac{5}{3}+\frac{1}{8}\sum_{j=3}^9(1-\gamma_j)+\varepsilon}, \\
\int_0^1 |F_2\cdots F_9| \mathrm{d}\alpha  \ll N^{\frac{5}{3}+\frac{1}{8}\sum_{j=2}^9(1-\gamma_j)+\varepsilon}.
\end{gather*}
Hence, according to (\ref{eq1_01}) and (\ref{eq3_2}), it is sufficient to prove that, for $i=1,2,\dots,9$, there holds
\begin{align*}
\sup_{\alpha\in[0,1]}|F_i-G|\ll N^{\frac{1}{3}(1-\Delta_i)-2\varepsilon},
\end{align*}
where
\begin{equation*}
\Delta_i =
\begin{cases}
\displaystyle\frac{3}{8}\sum\limits_{j=i+1}^9(1-\gamma_j), & \textrm{if $1\leqslant i\leqslant 8$,}  \\
\hspace{3em} 0, & \textrm{if \hspace{0.7em} $i=9$.}
\end{cases}
\end{equation*}
Accordingly,
from (\ref{eq3_1}), (\ref{eq3_2}) and the above discussions, it is easy to see that Theorem \ref{Theorem-1} is a corollary of the following proposition.

\begin{proposition}\label{Proposition3_1}
Assume that $1/2<\gamma<1$, $0\leqslant\Delta<1/2$, and
\begin{equation*}
\frac{80}{3}(1-\gamma) + \frac{80}{3} \Delta<1.
\end{equation*}
Then, uniformly for $\alpha\in[0,1]$, we have
\begin{equation}\label{eq3_4}
\frac{1}{\gamma}\sum_{\substack{p \leqslant N^{1/3} \\ p=[n^{1/\gamma}]}} p^{1-\gamma}(\log p)e(p^3\alpha)
=\sum_{p\leqslant N^{1/3}}(\log p)e(p^3\alpha)+O\big(N^{\frac{1}{3}(1-\Delta)-2\varepsilon}\big),
\end{equation}
where the implied constant may depend at most on $\gamma$, $\Delta$ and $\varepsilon$ only.
\end{proposition}

In the rest of this section, we shall transform the proof of Proposition \ref{Proposition3_1} to a special exponential sum estimate. We denote
\begin{gather*}
F(\alpha)=\frac{1}{\gamma}\sum_{\substack{p\leqslant N^{1/3}\\ p=[n^{1/\gamma}]}}
p^{1-\gamma}(\log p)e(p^3\alpha),
\end{gather*}
and
\begin{gather*}
G(\alpha)=\sum_{p\leqslant N^{1/3}}(\log p)e(p^3\alpha).
\end{gather*}
For $1/2<\gamma<1$, it is easy to see that
\begin{equation*}
\mathds{1}_{\gamma}(p):=[-p^\gamma]-[-(p+1)^\gamma]
=\begin{cases}
1, & \textrm{if $p=[n^{1/\gamma}]$ for some $n\in\mathbb{N}^+$}, \\
0, & \textrm{otherwise.}
\end{cases}
\end{equation*}
Then, we have
\begin{align*}
         F(\alpha)
= & \,\, \frac{1}{\gamma}\sum_{p\leqslant N^{1/3}}p^{1-\gamma}(\log p)\mathds{1}_{\gamma}(p)e(p^3\alpha)
                  \nonumber \\
= & \,\, \frac{1}{\gamma}\sum_{p\leqslant N^{1/3}}p^{1-\gamma}(\log p)
         \big([-p^{\gamma}]-[-(p+1)^{\gamma}]\big) e(p^3\alpha)
                  \nonumber \\
= & \,\, \frac{1}{\gamma}\sum_{p\leqslant N^{1/3}}p^{1-\gamma}(\log p)
         \big(((p+1)^{\gamma}-p^{\gamma})+(\psi(-(p+1)^{\gamma})-\psi(-p^{\gamma}))\big) e(p^3\alpha)
                  \nonumber \\
=:& \,\, \mathcal{S}_1+\mathcal{S}_2,
\end{align*}
where
\begin{align*}
         \mathcal{S}_1
= & \,\, \frac{1}{\gamma}\sum_{p\leqslant N^{1/3}}p^{1-\gamma}(\log p)
         \big((p+1)^{\gamma}-p^{\gamma}\big)e(p^3\alpha),
                  \nonumber \\
         \mathcal{S}_2
= & \,\, \frac{1}{\gamma}\sum_{p\leqslant N^{1/3}}p^{1-\gamma}(\log p)
         \big(\psi(-(p+1)^{\gamma})-\psi(-p^{\gamma})\big)e(p^3\alpha).
\end{align*}
For $\mathcal{S}_1$, we deduce that
\begin{align}\label{eq3_5}
         \mathcal{S}_1
= & \,\, \frac{1}{\gamma}\sum_{p\leqslant N^{1/3}} p^{1-\gamma} (\log p) p^{\gamma}
         \bigg(\bigg(1+\frac{1}{p}\bigg)^{\gamma}-1\bigg) e(p^3\alpha)
            \nonumber \\
= & \,\, \sum_{p\leqslant N^{1/3}}(\log p)e(p^3\alpha)+O\Bigg(\sum_{p\leqslant N^{1/3}} \frac{\log p}{p}\Bigg)
            \nonumber \\
= & \,\, G(\alpha)+O(\log N) .
\end{align}
By partial summation, we have
\begin{align}\label{S_2-error}
         \mathcal{S}_2
= & \,\, \frac{1}{\gamma} \sum_{p \leqslant N^{1/3}} p^{1-\gamma}(\log p)e(p^3\alpha)
         \big(\psi(-(p+1)^{\gamma})-\psi(-p^{\gamma})\big)
              \nonumber \\
= & \,\, \frac{1}{\gamma}\sum_{n\leqslant N^{1/3}}n^{1-\gamma}\Lambda(n)e(n^3\alpha) \big(\psi(-(n+1)^{\gamma})-\psi(-n^{\gamma})\big)+O(N^{1/2-\gamma/3}),
\end{align}
where the $O$--term on the right--hand side of (\ref{S_2-error}) is admissible since $2(1-\gamma)+2\Delta<1$. Thus,
it suffices to show that the following estimate
\begin{equation}\label{eq3_6}
\mathcal{S}(X)=\sum_{n\sim X}n^{1-\gamma}\Lambda(n)e(n^3\alpha)\big(\psi(-(n+1)^{\gamma})-\psi(-n^{\gamma})\big)  \ll X^{1-\Delta-6\varepsilon}
\end{equation}
holds for any $X\ll N^{1/3}$. Taking $H_0=X^{1-\gamma+\Delta+7\varepsilon}$, it follows from Lemma \ref{Finite-Fourier-expansion} that
\begin{align}
         \mathcal{S}(X)
= & \,\, \sum_{n\sim X}n^{1-\gamma}\Lambda(n)e(n^3\alpha)\sum_{0<|h|\leqslant H_0}
         \frac{e(h(n+1)^\gamma)-e(hn^\gamma)}{2\pi ih}
                \nonumber\\
  & \,\, +O\bigg(X^{1-\gamma}(\log X)\sum_{n\sim X}\min\bigg\{1,\frac{1}{H_0\|n^{\gamma}\|}\bigg\}\bigg). \label{eq3_7}
\end{align}
For the sum in the error term of \eqref{eq3_7}, by Lemma \ref{vander-Corput-method}, we deduce that
\begin{align}\label{eq3_8}
   & \,\, \sum_{n\sim X} \min\bigg\{1,\frac{1}{H_0 \|n^{\gamma}\|}\bigg\}
          =\sum_{n\sim X}\sum_{h=-\infty}^{\infty}a(h)e(hn^\gamma)
          \ll \sum_{h=-\infty}^{\infty}|a(h)|\Bigg|\sum_{n\sim X}e(hn^\gamma)\Bigg|
                 \nonumber \\
  \ll & \,\, \Bigg(\frac{X\log H_0}{H_0}+\sum_{0<|h|\leqslant H_0}\frac{1}{|h|}\Bigg|
             \sum_{n\sim X}e(hn^\gamma)\Bigg|+\sum_{|h|>H_0}\frac{H_0}{h^2}
             \Bigg|\sum_{n\sim X}e(hn^\gamma)\Bigg|\Bigg)
                 \nonumber \\
 \ll & \,\, \frac{X\log H_0}{H_0}+\sum_{0<|h|\leqslant H_0}\frac{1}{|h|}\big(|h|^{1/2}X^{\gamma/2}
            +|h|^{-1/2}X^{1-\gamma/2}\big)
                 \nonumber \\
     & \,\, +\sum_{|h|>H_0}\frac{H_0}{h^2}\big(|h|^{1/2}X^{\gamma/2}+|h|^{-1/2}X^{1-\gamma/2}\big)
                 \nonumber \\
 \ll & \,\, \bigg(\frac{X\log H_0}{H_0}+H_0^{1/2}X^{\gamma/2}+X^{1-\gamma/2}\bigg)
                 \nonumber \\
 \ll & \,\, X^{1-(1-\gamma)-\Delta-6\varepsilon}+X^{1/2+\Delta/2+4\varepsilon}+X^{1-\Delta/2}
                 \nonumber \\
 \ll & \,\, X^{1-(1-\gamma)-\Delta-6\varepsilon},
\end{align}
provided that $3(1-\gamma)+3\Delta<1$. According to (\ref{eq3_6}), (\ref{eq3_7}) and (\ref{eq3_8}), we only need to prove that, for $h\sim H\leqslant H_0$, there holds
\begin{equation}\label{eq3_9}
\sum_{h\sim H} \frac{1}{h}\bigg|\sum_{n\sim X} \Lambda(n) n^{1-\gamma} e(n^3\alpha) \big(e(h(n+1)^\gamma)-e(hn^\gamma)\big)\bigg|\ll X^{1-\Delta-6\varepsilon}.
\end{equation}
Set $H_1=X^{1-\gamma}$. When $H\leqslant H_1$, we write
\begin{equation*}
e(h(n+1)^\gamma)-e(hn^\gamma)=2\pi i h \gamma \int_0^1 (n+u)^{\gamma-1}e(h(n+u)^\gamma) \mathrm{d} u.
\end{equation*}
Then, we derive that
\begin{align*}
    & \,\, \sum_{h\sim H}\frac{1}{h}\bigg|\sum_{n\sim X}\Lambda(n)n^{1-\gamma}e(n^3\alpha)
           \big(e(h(n+1)^\gamma)-e(hn^\gamma)\big)\bigg|
                  \nonumber \\
\ll & \,\, \sum_{h\sim H}\bigg|\sum_{n\sim X}\Lambda(n)n^{1-\gamma}\int_0^1(n+u)^{\gamma-1}
           e\big(n^3\alpha+h(n+u)^\gamma\big)\mathrm{d}u\bigg|
                  \nonumber \\
\ll & \,\, \sum_{h\sim H}\max_{0\leqslant u\leqslant1}\bigg|\sum_{n\sim X}\Lambda(n)n^{1-\gamma}
           (n+u)^{\gamma-1}e\big(n^3\alpha+h(n+u)^\gamma\big)\bigg|
                  \nonumber \\
\ll & \,\, \max_{0\leqslant u\leqslant1}\sum_{h\sim H}\bigg|\sum_{n\sim X}
           \Lambda(n)e\big(n^3\alpha+h(n+u)^\gamma\big) \bigg|.
\end{align*}
When $H_1<H\leqslant H_0$, we treat these two terms separately. According to partial summation, we deduce that
\begin{align*}
    & \,\, \sum_{h\sim H}\frac{1}{h}\bigg|\sum_{n\sim X}\Lambda(n) n^{1-\gamma} e(n^3\alpha)
           \big(e(h(n+1)^\gamma)-e(hn^\gamma)\big)\bigg|
                  \nonumber \\
\ll & \,\, \frac{H_1}{H}\sum_{h\sim H}\bigg|\sum_{n\sim X}\Lambda(n)e(n^3\alpha)
           \big(e(h(n+1)^\gamma)-e(hn^\gamma)\big)\bigg|
                  \nonumber \\
\ll & \,\, \frac{H_1}{H}\sum_{h\sim H}\bigg|\sum_{n\sim X}\Lambda(n)e(n^3\alpha+h(n+1)^\gamma)\bigg|+
           \frac{H_1}{H}\sum_{h\sim H}\bigg|\sum_{n\sim X}\Lambda(n)e(n^3\alpha+hn^\gamma)\bigg|.
\end{align*}
Based on the above arguments, in order to prove (\ref{eq3_9}), it is sufficient to show that
\begin{equation}\label{eq3_10}
\mathscr{S}(X;\alpha,H,\gamma):=\min\bigg\{1,\frac{H_1}{H}\bigg\}\sum_{h\sim H}\bigg|\sum_{n\sim X} \Lambda(n)
e(n^3\alpha + h(n+u)^\gamma)\bigg|\ll X^{1-\Delta-6\varepsilon}
\end{equation}
holds for $0\leqslant u \leqslant 1$ and $H\leqslant H_0$, uniformly. The proof of (\ref{eq3_10}) will be given in Section
\ref{exponential}.

It is easy to see that (\ref{eq3_6}) follows from (\ref{eq3_7}), (\ref{eq3_8}), (\ref{eq3_9}) and (\ref{eq3_10}). Proposition \ref{Proposition3_1} now follows from (\ref{eq3_5}) and (\ref{eq3_6}).

\section{Estimations of exponential sums}\label{exponential}
Suppose that $1 / 2<\gamma<1$, $m \sim M$, $k \sim K$, $M K \asymp X$, $h$ is an integer satisfying $h \sim H \leqslant H_0=X^{1-\gamma+\Delta+7\varepsilon}$ and $H_1=X^{1-\gamma}$. Define
\begin{align*}
S_I(M,K;X) & :=\min\bigg\{1,\frac{H_1}{H}\bigg\}\sum_{h\sim H}\bigg|\sum_{m \sim M}\sum_{k \sim K}a_m
 e\big(m^3k^3\alpha+h(m k+u)^{\gamma}\big)\bigg|,
      \nonumber \\
S_{II}(M,K;X) & :=\min\bigg\{1,\frac{H_1}{H}\bigg\}\sum_{h\sim H}\bigg|\sum_{m\sim M}\sum_{k\sim K}a_mb_k
e\big(m^3k^3\alpha+h(m k+u)^{\gamma}\big)\bigg|,
\end{align*}
where $a_m$ and $b_k$ are complex numbers subject to $|a_m|\ll1$, $|b_k|\ll1$.
\begin{proposition}\label{Proposition4_1}
Assume that $a_m \ll 1$, and
\begin{equation}\label{eq4_1}
16(1-\gamma)+16\Delta<1.
\end{equation}
Let $MK\asymp X$, $1\ll M \ll X^\mathfrak{a}$, where
\begin{equation}\label{eq4_2}
\mathfrak{a}=\min\bigg\{\frac{3}{4}-\frac{7}{2}(1-\gamma)-\frac{15}{4}\Delta-27\varepsilon, \,\, 1-8(1-\gamma)-8\Delta-56\varepsilon\bigg\}.
\end{equation}
Then, we have
\begin{equation*}
S_{I}(M,K;X)\ll X^{1-\Delta-7\varepsilon}.
\end{equation*}
\end{proposition}
\begin{proof}
Let
\begin{equation*}
f(k)=m^3k^3\alpha+h(mk+u)^\gamma.
\end{equation*}
It is easy to see that
\begin{equation*}
f^{(4)}(k) \asymp H X^\gamma K^{-4}.
\end{equation*}
Hence, by Lemma \ref{Van-der-Corput-k-derivatives} with $q=2$ and $1 \leqslant H \leqslant H_0$,  we deduce that
\begin{align*}
             S_I(M,K;X)
  \ll & \,\, \min\bigg\{1,\frac{H_1}{H}\bigg\}\sum_{h\sim H}\sum_{m\sim M}\bigg|\sum_{k\sim K}
             e\big(m^3k^3\alpha+h(mk+u)^\gamma\big)\bigg|
                 \nonumber \\
\ll & \,\, \min\bigg\{1,\frac{H_1}{H}\bigg\}HM\big(K(HX^\gamma K^{-4})^{1/14}+K^{7/8}+K^{9/16}
           (HX^\gamma K^{-4})^{-1/8}\big)
                 \nonumber \\
\ll & \,\, \min\bigg\{1,\frac{H_1}{H}\bigg\}\big(H^{15/14}X^{5/7+\gamma/14}M^{2/7}
           +HX^{7/8}M^{1/8}+H^{7/8}X^{17/16-\gamma/8}M^{-1/16}\big)
                 \nonumber \\
\ll & \,\, H_1H_0^{1/14}X^{5/7+\gamma/14}M^{2/7}+H_1X^{7/8}M^{1/8}+H_1^{7/8}X^{17/16-\gamma/8}
           M^{-1/16}
                 \nonumber \\
\ll & \,\, X^{1-\Delta-7\varepsilon},
\end{align*}
where we use (\ref{eq4_1}) and (\ref{eq4_2}) in the last estimate.
\end{proof}

\begin{proposition}\label{Proposition4_2}
Assume that $a_m \ll 1$, $b_k \ll 1$. Let $MK\asymp X$, $X^\mathfrak{b} \leqslant M \leqslant X^\mathfrak{c}$, where
\begin{gather}
\mathfrak{b} = 16(1-\gamma)+16\Delta+112\varepsilon, \label{eq4_3}  \\
\mathfrak{c} = \min\bigg\{\frac{16}{17}-\frac{16}{17}(1-\gamma)-\frac{32}{17}\Delta-\frac{224}{17}\varepsilon, \,\,\, 2-32(1-\gamma)-32\Delta-224\varepsilon\bigg\}.\label{eq4_4}
\end{gather}
Then one has
\begin{equation*}
S_{II}(M,K;X) \ll X^{1-\Delta-7\varepsilon}.
\end{equation*}
\end{proposition}
\begin{proof}
For any fixed $u\in[0,1]$ and $H\leqslant H_0$, we define the complex number $c_h$ by means of the equation
\begin{equation*}
\bigg|\sum_{m\sim M}\sum_{k\sim K}a_m b_k e\big(m^3k^3\alpha+h(mk+u)^\gamma\big)\bigg|=
c_h\sum_{m\sim M}\sum_{k\sim K}a_mb_ke\big(m^3k^3\alpha+h(mk+u)^\gamma\big),
\end{equation*}
for $h\sim H$ and by taking $c_h=0$ otherwise. Plainly, one has $|c_h|=1$. Denote by $Q$ a parameter subject to
$Q\in[1,HK(\log X)^{-1}]$, which will be chosen later. We decompose the collection of available pairs $(h,k)$ into
sets $\mathscr{W}_q\,(1\leqslant q\leqslant Q)$, defined by
\begin{equation*}
\mathscr{W}_q=\bigg\{(h,k):\,h\sim H,\,\,k\sim K,\,\,\frac{4HK^\gamma(q-1)}{Q}<hk^\gamma\leqslant
\frac{4HK^\gamma q}{Q}\bigg\}.
\end{equation*}
Then, we obtain that
\begin{align*}
           \mathfrak{S}
 := & \,\, \sum_{h\sim H}\bigg|\sum_{m\sim M}\sum_{k\sim K}a_mb_k
           e\big(m^3k^3\alpha+h(m k+u)^{\gamma}\big)\bigg|
                  \nonumber \\
  = & \,\, \sum_{h\sim H}c_h\sum_{m\sim M}\sum_{k\sim K}a_mb_ke\big(m^3k^3\alpha+h(mk+u)^\gamma\big)
                  \nonumber \\
  = & \,\, \sum_{1\leqslant q\leqslant Q}\sum_{m\sim M}a_m\sum_{(h,k)\in\mathscr{W}_q}b_k c_h
           e\big(m^3k^3\alpha+h(mk+u)^\gamma\big).
\end{align*}
It follows from Cauchy's inequality and splitting argument that
\begin{align*}
            |\mathfrak{S}|^2
 \ll & \, QM\sum_{1\leqslant q\leqslant Q}\sum_{m\sim M}
            \Bigg|\mathop{\sum\sum}_{\substack{(h,k)\in\mathscr{W}_q\\ mk\sim X}}
            b_kc_he\big(m^3k^3\alpha+h(mk+u)^{\gamma}\big)\Bigg|^2
                    \nonumber \\
 \ll & \, QM \sum_{1\leqslant q\leqslant Q}\mathop{\sum\sum}_{(h_1,k_1)\in\mathscr{W}_q}
            \mathop{\sum\sum}_{(h_2,k_2)\in\mathscr{W}_q}
            \Bigg|\sum_{\substack{m\sim M\\ mk_1\sim X\\ mk_2\sim X}}
            e\big(m^3(k_1^3-k_2^3)\alpha+h_1(mk_1+u)^{\gamma}-h_2(mk_2+u)^{\gamma}\big)\Bigg|
                    \nonumber \\
\ll & \, QM \mathop{\sum_{h_1\sim H}\sum_{h_2\sim H}\sum_{k_1\sim K}\sum_{k_2\sim K}}_
           {|h_1k_1^\gamma-h_2k_2^\gamma |\leqslant 4HK^\gamma Q^{-1}}
           \Bigg|\sum_{\substack{m\sim M\\ mk_1\sim X\\ mk_2\sim X}} e\big(m^3 (k_1^3-k_2^3)\alpha
+ h_1(m k_1+u)^{\gamma} - h_2(m k_2+u)^{\gamma}\big)\Bigg|.
\end{align*}
Write
\begin{equation*}
f(m)= m^3(k_1^3-k_2^3)\alpha + h_1(m k_1+u)^{\gamma} - h_2(m k_2+u)^{\gamma},
\end{equation*}
and $\lambda = h_1k_1^\gamma-h_2k_2^\gamma$. It is easy to see that
\begin{align*}
f^{(4)}(m) & = \gamma(\gamma-1)(\gamma-2)(\gamma-3)\big(h_1k_1^4(m k_1+u)^{\gamma-4} - h_2k_2^4(m k_2+u)^{\gamma-4}\big) \\
& = \gamma(\gamma-1)(\gamma-2)(\gamma-3)m^{\gamma-4}(h_1k_1^\gamma-h_2k_2^\gamma) + O\bigg(\frac{H}{H_1M^4}\bigg) \asymp |\lambda|M^{\gamma-4},
\end{align*}
provided that $HH_1^{-1}M^{-\gamma}\ll|\lambda|\leqslant4HK^{\gamma}Q^{-1}$.

First, we consider the case $|\lambda|\ll HH_1^{-1}M^{-\gamma}$. If $H\leqslant H_1$, then $|\lambda|\ll M^{-\gamma}$, and we shall use the trivial estimate $M$ to bound the innermost sum over $m$. It is easy to see that the contribution of $M$ to $|\mathfrak{S}|^2$ is
\begin{align}\label{eq4_5}
\ll & \,\, QM^2\mathop{\sum_{h_1\sim H}\sum_{h_2\sim H}\sum_{k_1\sim K}\sum_{k_2\sim K}}_
           {|\lambda|\ll M^{-\gamma}} 1
\ll  QM^2\cdot\mathscr{N}(M^{-\gamma})
                    \nonumber \\
\ll & \,\, QM^2(M^{-\gamma}HK^{2-\gamma}+HK\log X)
                    \nonumber \\
\ll & \,\, QHX^{2-\gamma}+QHXM\log X\ll QHXM\log X,
\end{align}
provided that $M \gg X^{1-\gamma}$. If $H_1 < H \leqslant H_0$, we also use the trivial estimate $M$ to bound the innermost sum over $m$, and thus the contribution is
\begin{align}\label{eq4_6}
\ll & \,\,  QM^2 \mathop{\sum_{h_1\sim H}\sum_{h_2\sim H}\sum_{k_1\sim K}\sum_{k_2\sim K}}
            _{|\lambda|\ll HH_1^{-1}M^{-\gamma}} 1
\ll   QM^2\cdot \mathscr{N}(HH_1^{-1}M^{-\gamma})
                    \nonumber \\
\ll & \,\, QM^2(HH_1^{-1}M^{-\gamma}\cdot HK^{2-\gamma}+HK\log X)
                    \nonumber \\
\ll & \,\, QH^2H_1^{-1}X^{2-\gamma}+QHXM\log X
                    \nonumber \\
\ll & \,\, QHXM\log X,
\end{align}
provided that $M\gg X^{1-\gamma+\Delta+\varepsilon}$. Now, we consider the case with $\lambda$ being the magnitude
$HH_1^{-1} M^{-\gamma}\ll |\lambda| \leqslant 4HK^{\gamma}Q^{-1}$. By Lemma \ref{Van-der-Corput-k-derivatives} with $q=2$, we obtain
\begin{equation*}
\sum_{m\sim M} e(f(m)) \ll M^{5/7+\gamma/14}|\lambda|^{1/14}+M^{7/8}+M^{17/16-\gamma/8}|\lambda|^{-1/8}.
\end{equation*}
By noting that $Q\in [1,HK(\log X)^{-1}]$, it follows from Lemma \ref{Heath-Brown-1983} that the contribution of $M^{5/7+\gamma/14}|\lambda|^{1/14}$ to $|\mathfrak{S}|^2$ is
\begin{align}\label{eq4_7}
\ll & \,\, QM\cdot M^{5/7+\gamma/14}\mathop{\sum_{h_1\sim H}\sum_{h_2\sim H}\sum_{k_1\sim K}\sum_{k_2\sim K}}
           _{|\lambda|\leqslant 4HK^\gamma Q^{-1}} |\lambda|^{1/14}
                      \nonumber \\
\ll & \,\, QM^{12/7+\gamma/14} (HK^\gamma Q^{-1})^{1/14}\cdot\mathscr{N}(4HK^\gamma Q^{-1})
                      \nonumber \\
\ll & \,\, QM^{12/7+\gamma/14} (HK^\gamma Q^{-1})^{1/14}(H^2K^2Q^{-1}+HK\log X )
                      \nonumber \\
\ll & \,\, Q^{-1/14}H^{29/14}M^{12/7+\gamma/14}K^{2+\gamma/14}
                      \nonumber \\
\ll & \,\, Q^{-1/14}H^{29/14}X^{2+\gamma/14}M^{-2/7},
\end{align}
and the contribution of $M^{7/8}$ to $|\mathfrak{S}|^2$ is
\begin{align}\label{eq4_8}
\ll & \,\,  QM\cdot M^{7/8} \mathop{\sum_{h_1\sim H}\sum_{h_2\sim H}\sum_{k_1\sim K}\sum_{k_2\sim K}}
           _{|\lambda|\leqslant 4HK^\gamma Q^{-1}} 1
                      \nonumber \\
\ll & \,\,  QM^{15/8}\cdot\mathscr{N}(4HK^\gamma Q^{-1})
                      \nonumber \\
\ll & \,\,  QM^{15/8}(H^2K^2Q^{-1}+HK\log X)
                      \nonumber \\
\ll & \,\,  H^2M^{15/8}K^2
\ll   H^{2}X^{2}M^{-1/8}.
\end{align}
When $H_1<H\leqslant H_0$, it follows from the splitting argument and Lemma \ref{Heath-Brown-1983} that the contribution of $M^{17/16-\gamma/8}|\lambda|^{-1/8}$ to $|\mathfrak{S}|^2$ is
\begin{align}\label{eq4_9}
\ll & \,\, QM\cdot M^{17/16-\gamma/8}\mathop{\sum_{h_1\sim H}\sum_{h_2\sim H}\sum_{k_1\sim K}\sum_{k_2\sim K}}
           _{M^{-\gamma}\ll|\lambda|\leqslant4HK^\gamma Q^{-1}}|\lambda|^{-1/8}
                   \nonumber \\
\ll & \,\, (\log X)QM^{33/16-\gamma/8}\max_{M^{-\gamma}\leqslant\Delta\leqslant4HK^{\gamma}Q^{-1}}
           \mathscr{N}(2\Delta)\cdot\Delta^{-1/8}
                   \nonumber \\
\ll & \,\, (\log X)QM^{33/16-\gamma/8}\max_{M^{-\gamma}\leqslant \Delta \leqslant 4HK^{\gamma}Q^{-1}}
           (\Delta HK^{2-\gamma}+HK\log X)\Delta^{-1/8}
                   \nonumber \\
\ll & \,\, (\log X)^2QM^{33/16-\gamma/8}\big((HK^{\gamma}Q^{-1})^{7/8}HK^{2-\gamma}+M^{\gamma/8}HK\big)
                   \nonumber \\
\ll & \,\, (\log X)^2\big(Q^{1/8}H^{15/8}M^{33/16-\gamma/8}K^{2-\gamma/8}+QHM^{33/16}K\big)
                   \nonumber \\
\ll & \,\, (\log X)^2\big(Q^{1/8}H^{15/8}X^{2-\gamma/8}M^{1/16}+QHXM^{17/16}\big).
\end{align}
Similarly, when $H \leqslant H_1$, the contribution of $M^{17/16-\gamma/8}|\lambda|^{-1/8}$ to
$|\mathfrak{S}|^2$ is
\begin{align}\label{eq4_10}
\ll & \,\,  QM\cdot M^{17/16-\gamma/8}\mathop{\sum_{h_1\sim H}\sum_{h_2\sim H}\sum_{k_1\sim K}\sum_{k_2\sim K}}
            _{HH_1^{-1}M^{-\gamma}\ll|\lambda|\leqslant 4HK^\gamma Q^{-1}}|\lambda|^{-1/8}
                   \nonumber \\
\ll & \,\, (\log X)QM^{33/16-\gamma/8}\max_{HH_1^{-1}M^{-\gamma}\leqslant\Delta\leqslant4HK^{\gamma}Q^{-1}}
           \mathscr{N}(2\Delta)\cdot\Delta^{-1/8}
                   \nonumber \\
\ll & \,\, (\log X)QM^{33/16-\gamma/8}\max_{HH_1^{-1}M^{-\gamma}\leqslant \Delta \leqslant 4HK^{\gamma}Q^{-1}}
           (\Delta^{7/8}HK^{2-\gamma}+\Delta^{-1/8} HK\log X)
                   \nonumber \\
\ll & \,\, (\log X)^2QM^{33/16-\gamma/8}\big((HK^{\gamma}Q^{-1})^{7/8}HK^{2-\gamma}
           +(HH_1^{-1}M^{-\gamma})^{-1/8} HK\big)
                   \nonumber \\
\ll & \,\, (\log X)^2\big(Q^{1/8}H^{15/8}M^{33/16-\gamma/8}K^{2-\gamma/8}+QHM^{33/16}K \big)
                   \nonumber \\
\ll & \,\, (\log X)^2\big(Q^{1/8}H^{15/8}X^{2-\gamma/8}M^{1/16}+QHXM^{17/16}\big).
\end{align}
According to \eqref{eq4_5}--\eqref{eq4_10}, we derive that
\begin{align}\label{eq4_11}
(\log X)^{-2}|\mathfrak{S}|^2
 \ll & \,\, Q^{-1/14}H^{29/14}X^{2+\gamma/14}M^{-2/7} + H^{2}X^{2}M^{-1/8}
              \nonumber \\
     & \,\, +QHXM^{17/16} + Q^{1/8}H^{15/8}X^{2-\gamma/8}M^{1/16}.
\end{align}
By Lemma \ref{Optimization-Principle} and \eqref{eq4_11}, we choose an optimal $Q\in[1,HK]$ and deduce that
\begin{equation*}
 X^{-\varepsilon}\cdot|\mathfrak{S}|^2\ll H^2 X^{29/15+\gamma/15}M^{-47/240}+H^2X^2 M^{-1/8}
 +H^{15/8}X^{2-\gamma/8}M^{1/16} + HXM^{17/16}.
\end{equation*}
Hence, for $H\leqslant H_0$, we obtain that
\begin{align}\label{eq4_12}
           X^{-\varepsilon}\cdot |S_{II}(M,K;X)|
\ll & \,\, \min\bigg\{1, \frac{H_1}{H}\bigg\}\big(H X^{29/30+\gamma/30}M^{-47/480} + HX M^{-1/16}
               \nonumber \\
    & \hspace{7em}+H^{15/16}X^{1-\gamma/16}M^{1/32} + H^{1/2}X^{1/2}M^{17/32}\big)
               \nonumber \\
\ll & \,\, H_1 X^{29/30+\gamma/30}M^{-47/480}+H_1X M^{-1/16}
               \nonumber \\
    & \hspace{2em} + H_1^{15/16}X^{1-\gamma/16}M^{1/32} + H_1^{1/2}X^{1/2}M^{17/32},
\end{align}
which combined with \eqref{eq4_3} and \eqref{eq4_4} yields
\begin{equation*}
S_{II}(M,K;X)\ll X^{1-\Delta-7\varepsilon}.
\end{equation*}
This completes the proof of Proposition \ref{Proposition4_2}.
\end{proof}

\section{Proof of Proposition 3.1}
Suppose that $\gamma\in(1/2,1)$ and $\Delta\in[0,1/2)$, which satisfy the condition of Proposition \ref{Proposition3_1}, i.e.,
\begin{equation*}
\frac{80}{3}(1-\gamma)+\frac{80}{3}\Delta<1.
\end{equation*}
Let
\begin{align*}
         \mathfrak{a}
= & \,\, \min\bigg\{\frac{3}{4}-\frac{7}{2}(1-\gamma)-\frac{15}{4}\Delta-27\varepsilon,
         1-8(1-\gamma)-8\Delta-56\varepsilon\bigg\},
               \nonumber \\
\mathfrak{b} = & \,\, 16(1-\gamma)+16\Delta+112\varepsilon,
                \nonumber \\
\mathfrak{c} = & \,\, \min\bigg\{\frac{16}{17}-\frac{16}{17}(1-\gamma)-\frac{32}{17}\Delta-\frac{224}{17}\varepsilon, \,\, 2-32(1-\gamma)-32\Delta-224\varepsilon\bigg\}.
\end{align*}
Then, we have
\begin{equation*}
\mathfrak{b}<2/3,\quad 1-\mathfrak{c}<\mathfrak{c}-\mathfrak{b},\quad \mathfrak{b}<\mathfrak{a}.
\end{equation*}
According to Heath--Brown's identity (Lemma \ref{Heath-Brown-identity} with $k=3$), $\mathscr{S}(X;\alpha,H,\gamma)$ can be written as linear combination of $O(\log^6 X)$ sums, each of which is of the form
\begin{align*}
\mathcal{T}^*(K_1,\dots,K_6)
:= & \,\, \min\bigg\{1,\frac{H_1}{H}\bigg\}\times\sum_{h\sim H}\Bigg|
       \mathop{\sum_{k_1\sim K_1}\!\!\cdots\!\!\sum_{k_6\sim K_6}}_{k_1\dots k_6 \sim X}(\log k_1)\mu(k_4)\mu(k_5)\mu(k_6)
          \nonumber \\
   & \,\, \hspace{9em}\times e\big((k_1\dots k_6)^3\alpha + h(k_1\dots k_6)^{\gamma}\big)\Bigg|,
\end{align*}
where $K_1\dots K_6\asymp X$, $K_j\leqslant (2X)^{1/3},\,(j=4,5,6)$, and some $k_i$ may only take value $1$. Therefore, it is sufficient for us to give upper bound estimates as follows
\begin{equation*}
\mathcal{T}^*(K_1,\dots,K_6) \ll X^{1-\Delta-7\varepsilon}.
\end{equation*}
Next, we shall consider three cases.

\noindent
\textbf{Case 1.} If there exists a $K_i$ such that $K_i \geqslant (2X)^{1-\mathfrak{b}}$, then there must hold $i\leqslant3$ for the fact that $K_j\ll X^{1/3}$ with $j=4,5,6$. Let
\begin{equation*}
m=\prod_{\substack{1\leqslant j\leqslant6\\ j\neq i}} k_j, \quad k=k_i,
\quad M=\prod_{\substack{1\leqslant j\leqslant6\\j\neq i}} K_j, \quad K=K_i.
\end{equation*}
In this case, we can see that $\mathcal{T}^*(K_1,\dots,K_6)$ is a sum of ``Type I'', i.e., $S_{I} (M,K;X)$, subject
to $1\ll M \ll X^\mathfrak{b}\ll X^\mathfrak{a}$. By proposition \ref{Proposition4_1}, we derive that
\begin{equation*}
\mathcal{T}^*(K_1,\dots,K_6)\ll X^{1-\Delta-7\varepsilon}.
\end{equation*}

\noindent
\textbf{Case 2.} If there exists a $K_i$ such that $X^{1-\mathfrak{c}}\ll K_i\ll X^{1-\mathfrak{b}}$, then we take
\begin{equation*}
m=\prod_{\substack{1\leqslant j\leqslant6\\ j\neq i}} k_j, \quad k=k_i,
\quad M=\prod_{\substack{1\leqslant j\leqslant6\\ j\neq i}} K_j, \quad K=K_i.
\end{equation*}
Thus, $\mathcal{T}^*(K_1,\dots,K_6)$ is a sum of ``Type II'', i.e., $S_{II} (M,K;X)$, subject to
$X^\mathfrak{b}\ll M\ll X^\mathfrak{c}$. By Proposition \ref{Proposition4_2}, we derive that
\begin{equation*}
\mathcal{T}^*(K_1,\dots,K_6)\ll X^{1-\Delta-7\varepsilon}.
\end{equation*}

\noindent
\textbf{Case 3.} If $K_i<X^{1-\mathfrak{c}}\, (i=1,2,\dots,6)$, without loss of generality, we postulate that $K_1\geqslant K_2\geqslant\dots\geqslant K_6$. Denote by $\ell$ the natural number such that
\begin{equation*}
K_1\dots K_{\ell-1}<X^{1-\mathfrak{c}}, \qquad\quad K_1\dots K_{\ell}\geqslant X^{1-\mathfrak{c}}.
\end{equation*}
Since $K_1<X^{1-\mathfrak{c}}$ and $K_6<X^{1-\mathfrak{c}}$, then $2\leqslant\ell\leqslant5$. Therefore, there holds
\begin{equation*}
X^{1-\mathfrak{c}}\leqslant K_1\dots K_{\ell}=(K_1\dots K_{\ell-1})\cdot K_{\ell}<
X^{1-\mathfrak{c}}\cdot X^{1-\mathfrak{c}}< X^{1-\mathfrak{b}}.
\end{equation*}
Let
\begin{equation*}
m=\prod_{j= \ell+1}^6 k_j, \quad k=\prod_{j=1}^{\ell} k_j, \quad M=\prod_{j= \ell+1}^6 K_j, \quad K= \prod_{j=1}^{\ell}K_j.
\end{equation*}
At this time, $\mathcal{T}^*(K_1,\dots,K_6)$ is a sum of ``Type II'', i.e., $S_{II} (M,K;X)$, subject to
$X^\mathfrak{b}\ll M\ll X^\mathfrak{c}$. By proposition \ref{Proposition4_2}, we derive that
\begin{equation*}
\mathcal{T}^*(K_1,\dots,K_6)\ll X^{1-\Delta-7\varepsilon}.
\end{equation*}
Consequently, we obtain
\begin{equation*}
\mathscr{S}(X;\alpha,H,\gamma)\ll X^{1-\Delta-7\varepsilon}\cdot (\log X)^6 \ll X^{1-\Delta-6\varepsilon}.
\end{equation*}
This completes the proof of Proposition \ref{Proposition3_1}.

\section*{Acknowledgement}

The authors would like to appreciate the referee for his/her patience in refereeing this paper.
This work is supported by Beijing Natural Science Foundation (Grant No. 1242003), and
the National Natural Science Foundation of China (Grant Nos. 12471009, 12301006, 11901566, 12001047).

\end{document}